\documentclass[12pt]{amsart}

\usepackage{etex}
\usepackage[english]{babel}
\usepackage[cp1250]{inputenc}
\usepackage{amssymb,amsthm,amsfonts,amsmath}
\usepackage{mathrsfs}
\usepackage{verbatim}
\usepackage{ulem}
\usepackage{mathtools}

\usepackage{url}
\usepackage{amsopn}
\usepackage{graphicx}
\usepackage[usenames, dvipsnames]{xcolor}
\DeclareGraphicsExtensions{.pdf,.png,.jpg}

\usepackage[shortlabels]{enumitem}
\usepackage{titletoc}
\definecolor{darkblue}{rgb}{0.0, 0.0, 0.55}
\usepackage[colorlinks,linkcolor=BrickRed,citecolor=OliveGreen,urlcolor=darkblue,hypertexnames=true]{hyperref}
\usepackage[a4paper,margin=2.5cm]{geometry}
\usepackage[baseline]{euflag}

\DeclareMathOperator{\btv}{btv}
\DeclareMathOperator{\symm}{Diff}

\DeclareMathOperator{\coPos}{COP}
\DeclareMathOperator{\cop}{COP}
\DeclareMathOperator{\COP}{COP}
\DeclareMathOperator{\DNN}{DNN}
\DeclareMathOperator{\psd}{PSD}

\DeclareMathOperator{\SPN}{SPN}
\DeclareMathOperator{\nn}{NN}

\DeclareMathOperator{\Sos}{SOS}
\DeclareMathOperator{\Vol}{Vol}

\DeclareMathOperator{\CP}{CP}
\DeclareMathOperator{\cp}{CP}

\DeclareMathOperator{\Tr}{tr}

\DeclareMathOperator{\vrad}{vrad}

\numberwithin{equation}{section}

\newcommand{\x}{{\tt x}}

\newcommand{\Sym}{\mathbb{S}}

\newcommand{\RR}{\mathbb R}

\newcommand{\NN}{\mathbb N}

\newcommand{\cB}{\mathcal B}

\newcommand{\cH}{\mathcal H}

\newcommand{\dd}{{\rm d}}

\newtheorem{theorem}{Theorem}[section]

\theoremstyle{definition}
\newtheorem{definition}[theorem]{Definition}
\newtheorem{remark}[theorem]{Remark}

\makeatletter
\newtheorem*{rep@thm}{\rep@title}
\newcommand{\newreptheorem}[2]{%
	\newenvironment{rep#1}[1]{%
		\def\rep@title{#2 \ref{##1}}%
		\begin{rep@thm}}%
		{\end{rep@thm}}}
\makeatother
\newreptheorem{thm}{Theorem}

\title[A random copositive matrix is CP with positive probability]{
A random copositive matrix is completely positive with positive probability}

\author[T.\ \v Strekelj]{Tea \v Strekelj${}^1$}
\address{Tea \v Strekelj, Famnit, University of Primorska, Koper \&
Institute of Mathematics, Physics and Mechanics, Ljubljana, Slovenia
   }
   \email{tea.strekelj@famnit.upr.si}
\thanks{${}^Q$This work was performed within the project COMPUTE, funded within the QuantERA II Programme that has received funding from the EU's H2020 research and innovation programme under the GA No 101017733 {\normalsize\euflag}}

\author[A. Zalar]{Alja\v z Zalar${}^{2,Q}$}
\address{Alja\v z Zalar, 
Faculty of Computer and Information Science, University of Ljubljana  \& 
Faculty of Mathematics and Physics, University of Ljubljana  \&
Institute of Mathematics, Physics and Mechanics, Ljubljana, Slovenia.}
\email{aljaz.zalar@fri.uni-lj.si}
\thanks{${}^1$Supported by the Slovenian Research Agency grant J1-60011.}
\thanks{${}^2$Supported by the Slovenian Research Agency 
program P1-0288 and grants J1-50002, J1-3004, J1-60011.}

\subjclass[2020]{13J30, 47L07, 52A40 (Primary); 90C22, 90C27 (Secondary)}
\date{\today}
\keywords{copositive matrix, completely positive matrix, positive polynomial, sum of squares, convex cone}

\title{Construction of exceptional copositive matrices}

\begin{document}
\setcounter{tocdepth}{3}
\contentsmargin{2.55em}
\dottedcontents{section}[3.8em]{}{2.3em}{.4pc}
\dottedcontents{subsection}[6.1em]{}{3.2em}{.4pc}
\dottedcontents{subsubsection}[8.4em]{}{4.1em}{.4pc}

%\makeatletter
%\newcommand{\mycontentsbox}{%
%{%\centerline{NOT FOR PUBLICATION}
%\addtolength{\parskip}{.3pt}
%\tableofcontents}}
%\def\enddoc@text{\ifx\@empty\@translators \else\@settranslators\fi
%\ifx\@empty\addresses \else\@setaddresses\fi
%\newpage\mycontentsbox\newpage\printindex}
%\makeatother

\begin{abstract}
	An $n\times n$ symmetric matrix $A$ is copositive if the quadratic form $x^TAx$	
	is nonnegative on the nonnegative orthant $\RR^{n}_{\geq 0}$. 
  	The cone of copositive matrices contains the cone of matrices which are the sum of a positive semidefinite matrix and a nonnegative one and the latter contains the cone of 
        completely positive matrices. These are the matrices
        of the form $BB^T$ for some $n\times r$ matrix $B$ with nonnegative entries. The above inclusions are strict for $n\geq5.$
        The first main result of this article is a free probability inspired construction of exceptional 
        copositive matrices of all sizes $\geq 5$, i.e., copositive matrices that are not the sum of a positive semidefinite matrix and a nonnegative one.
        The second contribution of this paper addresses the asymptotic ratio of the volume radii of compact sections of the cones of copositive and completely positive matrices. In a previous work by the authors, it was shown that, by identifying symmetric matrices naturally with quartic even forms, and equipping them with the $L^2$ inner product and the Lebesgue measure, the ratio of the volume radii of sections with a suitably chosen hyperplane is bounded below by a constant independent of $n$ as $n$ tends to infinity. In this paper, we extend this result by establishing an analogous bound when the sections of the cones are unit balls in the Frobenius inner product.
        %Consequently, the same holds true for the ratio
        %of volume radii of any two cones sandwiched between them,
        %e.g., the cones of positive semidefinite matrices, matrices with nonnegative entries, %their intersection and their Minkowski sum.\looseness=-1
\end{abstract}

\maketitle

\section{Introduction}

Copositive and completely positive matrices have gained considerable attention in recent years. They appear in combinatorial analysis, computational mechanics,  dynamical systems, control theory and especially in optimization.  
 This is because many combinatorial and nonconvex quadratic optimization problems can be formulated as linear problems over the larger cone of copositive or the smaller cone of completely positive matrices \cite{KP02,Bur09,RRW10,DR21}.
In this article we streamline the proof of the main result of \cite{KSZ}, which compares the asymptotic volumes of these two cones of matrices. Moreover, we give an explicit construction of \textit{exceptional matrices}, i.e., matrices that belong to the larger cone, but not to the smaller one.

\subsection{Notation}
    For $n\in \NN$ denote by $M_n(\RR)$ the $n\times n$ real matrices and let $\Sym_n=\left\{ A\in M_n(\RR)\colon A^T=A\right\}$ be its subspace of real symmetric matrices,
    where $^T$ stands for the usual transposition of matrices.
    Denote by $\RR[\x]$ be the vector space of real polynomials in the variables $\x=(x_1,\ldots,x_n)$ and let
     $\RR[\x]_k$ be its subspace of \textbf{forms of degree} $k$, i.e., 
    homogeneous polynomials from $\RR[\x]$ of degree $k$.
    To any matrix 
        $A=[a_{ij}]_{i,j=1}^n\in \Sym_n$ 
    we associate the quadratic form 
    \begin{equation}
	\label{correspondence-quadratic}
		p_A(\x):=\x^T A\x=\sum_{i,j=1}^n a_{ij}x_ix_j\in \RR[\x]_2.
	\end{equation}

\subsection{Basic definitions and background}
This article  studies the inclusion properties of the cones of the following classes of matrices.

\begin{definition}
\label{def-cones}
    A matrix $A\in \Sym_n$ is:
    \begin{enumerate}
    \item 
    \textbf{copositive} 
    if $p_A$ is nonnegative  
	on the nonnegative orthant 
		$$\RR^{n}_{\geq 0}:=\{(x_1,\ldots,x_n)\colon x_i\geq 0,\ i=1,\ldots,n\},$$
	i.e., $p_A(\x)\geq 0$ for every $\x\in\RR^{n}_{\geq 0}$.
    Equivalently, $A$ is copositive iff the quartic form 
     \begin{equation}\label{eq:q_A}
     q_A(\x):=p_A(x_1^2,\ldots,x_n^2)\in \RR[\x]_4
     \end{equation}
    is nonnegative on $\RR^n$.        
    We write $\coPos_n$ for the cone of all $n\times n$ copositive matrices.
    \item
    \textbf{positive semidefinite (PSD)}	 
        if all of its eigenvalues are nonnegative.
    Equivalently, $A$ is PSD 
        iff $p_A(\x)\geq 0$ for all $\x\in \RR^n$
        iff $A=BB^T$ for some matrix $B\in M_n(\RR)$.
	We write $A\succeq 0$ to denote that $A$ is PSD and $\psd_n$ stands for the cone of all $n\times n$ PSD matrices.
    \item   
        \textbf{nonnegative (NN)}
        if all of its entries are nonnegative, i.e., 
        $A=[a_{ij}]_{i,j=1}^n$ with 
        $a_{ij}\geq 0$ for $i,j=1,\ldots,n$.
    We write $\nn_n$ for the cone of all $n\times n$ NN matrices.
    \item 
    \textbf{SPN}
    (sum of a positive semidefinite matrix and a nonnegative one)
        if it is of the form $A=P+N$, 
        where 
            $P\in \psd_n$
        and 
            $N\in \nn_n$.
    We write $\SPN_n:=\psd_n+\nn_n$ for the cone of all $n\times n$ SPN matrices.
    \item   
    \textbf{doubly nonnegative (DNN)}
        if it is PSD and NN.
    We use $\DNN_n:=\psd_n\cap \nn_n$ for the cone of all $n\times n$ DNN matrices.
    \item   
    \textbf{completely positive (CP)}{\footnote{Despite the similar name, the CP matrices considered here are not related to the CP maps ubiquitous in operator algebra \cite{Pau02}.}}
        if $A=BB^T$ for some $r\in\NN$ and $n\times r$ entrywise nonnegative matrix $B$.
    We write $\cp_n$ for the cone of all $n\times n$ CP matrices.
    \end{enumerate}
\end{definition}

 The presented matrices clearly form the following chain of inclusions:
    \begin{equation}
        \label{101222-1515}
            \coPos_n
            \supseteq 
                \SPN_n
            \supseteq 
                \psd_n\cup\nn_n
            \supseteq
                \DNN_n
            \supseteq
                \cp_n.
    \end{equation}

After formulating a combinatorial problem as a conic linear problem over $\cop_n$ or $\cp_n,$ the complexity of the problem is reduced to the constraints of the respective cone. However, the membership problem for $\COP_n$ is co-NP-complete \cite{MK87} and NP-hard for $\CP_n$ \cite{DG14}.
For this reason, Parrilo \cite{Par00} proposed an increasing hierarchy of  cones
$K_n^{(r)}:=\{A\in \Sym_n\colon (\sum_{i=1}^n x_i^2)^r \cdot p_A(\x) \text{ is a sum of squares of forms}\},$ which give a tractable inner approximation of the cone $\COP_n$ based on semidefinite programming.
Clearly, 
\begin{equation}
\label{Parrilo-approx}
	\bigcup_{r\in \NN_0}  K_n^{(r)}\subseteq \COP_n,
\end{equation}
and 
 a result of P\'olya \cite{Pol28} gives a statement on the quality of the approximation, namely
	$\mathrm{int}(\COP_n)\subseteq \bigcup_{r\in \NN_0} K_n^{(r)}.$
It was shown in \cite[p.\ 63--64]{Par00} that $K_n^{(0)}=\SPN_n$. 
Also, $\SPN_n=\COP_n$ for $n\leq 4$  by \cite{MM62}. Whence, $K_n^{(0)}=\COP_n$
and the inclusion in \eqref{Parrilo-approx} is in fact equality for $n\leq 4$ (see also \cite{Dia62}).
However, for $n\geq 5$, the cone $K_n^{(0)}$ is strictly contained in $\COP_n.$ For $n=5,$ the strict inclusion is testified by the so-called Horn matrix \cite{HN63}

\begin{equation}
\label{Horn-matrix}
H=
\begin{pmatrix*}[r]
	1 & -1 & 1 & 1 & -1 \\
	-1 & 1 & -1 & 1 & 1 \\
	1 & -1 & 1 & -1 & 1 \\
	1 & 1 & -1 & 1 & -1 \\
	-1 & 1 & 1 & -1 & 1 
\end{pmatrix*}
\end{equation}
giving a standard example of a copositive matrix that is not $\SPN$.

Further, $H\in K_5^{(1)}$ by \cite{Par00}, but $\COP_5\neq K_5^{(r)}$ for any $r\in \NN$ by \cite{DDGH13}.
Very recently it was shown in \cite{LV22b,SV+} that for $n=5,$ the inclusion in \eqref{Parrilo-approx} is also equality,
while for $n\geq 6,$ the inclusion becomes strict \cite{LV22a}. The aim of this paper is first  to find further testimonies of the gap between $\cop$ and $\cp$ matrices. Next, we imply that this construction cannot be randomized by showing that the
asymptotic ratio of the volume radii of compact sections of the cones $\cop$ and $\cp$ is strictly positive as $n$ goes to infinity.
A nice exposition on the classes of matrices defined above can be found in \cite{BSM21} and
some open problems regarding $\COP$ and $\CP$ are presented in \cite{BDSM15}.

\subsection{Main results} \label{sec: main}

The first main result is a bootstrap method to find exceptional doubly nonnegative (e-DNN) matrices, i.e., doubly nonnegative matrices that are not completely positive. We first find a seed e-DNN matrix of size $5\times 5,$ which then gives rise to a family of e-DNN matrices of arbitrary sizes $\geq 5.$

The construction is inspired by the free probability construction in \cite{CHN17} of positive maps between matrix spaces that are not completely positive. For each $f \in L^\infty [0,1]$ consider the corresponding multiplication operator $M_f$ on $L^2[0,1];$ that is
$$
M_f\, g = fg
$$
for $g \in L^2[0,1].$
With respect to the standard orthonormal basis for $L^2[0,1]$ given by
\begin{equation}
        \label{basis-1}
        \cB:=  \big\{1 \big\} \cup
        \big\{\sqrt{2}\cos(2k\pi x)\colon k\in \NN \big\}\cup 
        \big\{\sqrt{2}\sin(2k\pi x)\colon k\in \NN\big\},
    \end{equation}
 each such multiplication operator can be represented by an infinite matrix. 
 For a closed subspace $\cH\subseteq L^2[0,1]$  denote by $P_\cH:L^2[0,1]\to \cH$ the orthogonal projection onto $\cH$. Then for any $f \in \cH,$ the operator $M_f^{\cH}:=P_\cH M_fP_\cH$ is in fact a multiplication operator on $\cH$ and can be as well represented by a (possibly infinite) matrix.
    Our idea is to find an infinite dimensional $\cH$ and an $f \in \cH$ such that
    $M_f^{\cH}$
    has all 
    finite principal submatrices
    $\DNN$ but not $\CP$.

\subsubsection{Construction of exceptional DNN matrices of all sizes $\geq 5$}
\label{construction-basic-idea}

The setting in which we work is the following:
    \begin{align}
        \label{form-of-f}
        &f \text{ is of the form }
           1+ 2 \sum_{k=1}^m a_k\cos(2k\pi x ),\quad 
            m\in \NN,\; a_1\geq 0,\ldots,a_m\geq 0,\\
        \nonumber
        &\cH\subseteq L^2[0,1]
        \text{ is spanned by the functions }\cos(2k\pi x),k\in \NN_0.
    \end{align}
    
    For $n\in \NN$, let $\cH_n$ be the finite-dimensional subspace of $\cH$
    spanned by the functions 
        $$1,\sqrt{2}\cos(2\pi x),\ldots,                  
            \sqrt{2}\cos(2(n-1)\pi x)$$
    and let $P_n:\cH\to \cH_n$ the orthogonal projection onto $\cH_n$.
    Clearly, all the matrices
    \begin{equation}
        \label{candidates-DNN-not-CP}
            A^{(n)}:=P_nM_f^\cH P_n,\; n\in\NN
    \end{equation}
    are $\nn$ since $f$ has positive Fourier coefficients. 
    To certify that they 
    are $\psd$, 
    we impose the condition that $f$ is a sum of squares ($\Sos$) of trigonometric polynomials, i.e., 
    \begin{equation}
        \label{f-psd-certify}
        f=v^TBv, \quad \text{where}\quad B\in\psd_{m'+1}
    \end{equation}
    and
    \begin{equation*}
        \label{vector-v}
        v^T
        =\begin{pmatrix}
            1 & \cos(2\pi x) & \cdots & \cos(2m'\pi x)
        \end{pmatrix}\quad\text{for some }m'\leq m.
    \end{equation*}
    Finally, to achieve that $A^{(n)}\not\in \CP_n$ for $n\geq 5$, we demand that
    \begin{equation}
        \label{not-cp}
            \langle A^{(5)},H\rangle<0,
    \end{equation}
    where $H$ is the Horn matrix of \eqref{Horn-matrix} and $\langle\cdot,\cdot\rangle$ is the usual Frobenius inner product on symmetric matrices, i.e., 
    $\langle A,B\rangle=\Tr(AB)$. Since  CP matrices are dual to the COP matrices w.r.t.\ the Frobenius inner product, this condition indeed certifies that $A^{(n)}\not\in \CP_n$ for all $n$ as we explain in Subsection \ref{subsec-findex}.

Now let $m=6$. The above construction can be implemented via the following feasibility SDP\looseness=-1
\begin{equation}
\begin{split}\label{eq-sdp}
	&\Tr(A^{(5)}H) = -\epsilon,\\ 
	%&A^{(5)} \succeq 0,\quad
        %\begin{color}{red} \text{Is not this automatic from }B\succeq 0?
        %\end{color}\\	
	&f = v^{\text{T}}Bv \quad \text{with}\quad B\succeq 0,\\
	&a_i \geq 0, \quad  i = 1, \ldots, 6,
\end{split}
\end{equation}
where $\epsilon>0$ is predetermined (small enough).
Solving \eqref{eq-sdp} for  different values of $\epsilon$ and 
$m^\prime \leq 6,$ Mathematica's semidefinite optimization solver gives an exceptional DNN matrix $A^{(5)}$ (see Subsection \ref{subsec-ex} for an explicit example).
 We remark that the idea is to search for an $f$ as in \ref{form-of-f} with the smallest $m$ as possible to reduce the complexity of the SDP \ref{eq-sdp}.
 The choice of $m=6$ seemed optimal from our experiments.

\subsubsection{Construction of exceptional copositive matrices of all sizes $n\geq 5$}
    To construct exceptional copositive matrices of arbitrary size we proceed as follows. For $n \geq 5$ let $A^{(n)}$ be a DNN matrix constructed by the above procedure.
    To obtain an exceptional copositive matrix $C$ of size $n\times n$ we {impose the conditions}
    \begin{align}
        \label{not-SPN}
        \begin{split}
        &\langle A^{(n)},C\rangle<0,\\
        &\big(\sum_{i=1}^nx_i^2\big)^kq_C \text{ is SOS for some }k\in \NN
        \end{split}
    \end{align}
    with $q_C$ as in \eqref{eq:q_A}.
    Searching for $C$ satisfying \eqref{not-SPN}
    for fixed $k$
    can again be formulated as a feasibility SDP. For an explicit example obtained in this way see Subsection \ref{subsec-ex}.\looseness=-1

\subsubsection{Second main result}

Let $V$ be a finite-dimensional Hilbert space equipped with the pushforward measure of the Lebesgue measure $\mu$ on $\RR^{\dim V}$. 
A natural way to compare the volumes of two cones $K_1$, $K_2$ in $V$ is to compare the compact sections of both cones when intersected with some ``fair" subset of $V$. A seemingly fair choice  is the unit ball $B$ of $V$. In this case the task is to derive an estimate for 
the so--called
\textbf{ball-truncated volume of $K_i$} \cite{ST15a}      $$\btv(K_i):=\Vol(K_i\cap B),$$ 
where the volume $\Vol$ is computed with respect to $\mu$. 
If one is interested only in the asymptotical behaviour of the volume difference, then comparing 
\textbf{volume radii} $\vrad(K_i\cap B)$, defined by \looseness=-1
	$$\vrad(K_i\cap B)=\left(\frac{\btv(K_i\cap B)}{\Vol(B)}\right)^{1/\dim V},$$
is equally informative (see also \cite[Remark 2.4]{KSZ} for a detailed discussion on the ratio of volumes versus the ratio of volume radii).

Let $B_n\subseteq \Sym_n$ be the unit ball w.r.t.~the Frobenius inner product.
Our second main result compares the sizes of the convex cones $K$ from Definition \ref{def-cones} 
by comparing the volumes of their intersections 
with $B_n$,
i.e., $K^{(B_n)}:=K\cap B_n$.

\begin{repthm}{vrad-of-our-sets} 
    \label{intro-vrad-of-our-sets} 
    We have that
    \begin{align*}
        %\label{intro-vrad-asymptotics}
            \frac{1}{8\sqrt{2}}
            &\leq 
            \vrad(\CP^{(B_n)}_n)
            \leq
            \vrad(\nn^{(B_n)}_n)
            =\frac{1}{2}
            \leq 
            \vrad(\COP^{(B_n)}_n)
            \leq 
            1.
    \end{align*}
    In particular,
    $$
    \frac{1}{8\sqrt{2}}
    \leq
    \frac
    {\vrad(K^{(B_n)}_n)}{\vrad(\COP^{(B_n)}_n)}
    \leq 
    1,
    $$
    where $K\in \{\CP,\DNN,\psd,\nn,\SPN\}$.
\end{repthm}

\begin{remark}
\begin{enumerate}
\item 
Deriving tight estimates for the ball-truncated volume $\btv(K^{(B)})$ of a cone $K$ is very demanding and infeasible for most cones in dimensions beyond 3. This is due to the fact that the conditions defining the section $K^{(B)}$ are quadratic in the coordinates of $\RR^{\dim V}$. 
To compensate on this problem one approach is to compare the cones when intersected with a suitably chosen half-space or equivalently, a hyperplane.
For $x\in V$ let
$H_x:=\{u\in V\colon \langle x,u\rangle \leq 1\}$
be a closed halfspace and 
$\partial H_x:=\{u\in V\colon \langle x,u\rangle= 1\}$ its boundary hyperplane.
\begin{enumerate}
    \item 
    A comparison of the sizes of the cones can be done in terms of their \textbf{least partial volumes} \cite[Definition 1.1]{ST15a}. The least partial volume of a cone $K$ is the smallest volume of $K \cap H_x$,
    where $x$ runs over the unit sphere in $V$.
\item When we want to compare the cones $K_1$, $K_2$ that are dual to each other in the inner product on $V$, i.e., 
$$K_2=K_1^\ast=
\{u\in V\colon \langle x,u\rangle \geq 0
\text{ for all }x\in K_1\},$$
a ``fair" choice of $x$ in $H_x$ is the so-called \textbf{volumetric center} $\rho(K_1)$ of $K_1$ \cite[Definition 1.2]{ST15a}. This is because $\rho(K_1)$ is the centroid of $K_2\cap \partial H_{\rho(K_1)}$ (see \cite[p.\ 2 and Lemma 5.2]{ST15b}), which is equivalent to the fact that $\rho(K_1)$ is the so--called Santal\'o point of $K_1\cap \partial H_{\rho(K_1)}$, for which the Blaschke--Santalo inequality \cite[p.\ 90]{MP90} can be applied.
\end{enumerate}
\item 
In our previous work \cite{KSZ} on estimating
the quantitative gap between $\COP_n$ and $\CP_n$
we used the identification with quartic even forms
$\eqref{eq:q_A}$ and then compared volumes of the corresponding cones in positive even quartics. The inner product is taken to be the $L^2$ norm, i.e., 
$\langle f,g\rangle = \int_{S^{n-1}} fg\; \dd\sigma$, where $\sigma$ is the rotation invariant probability measure on the unit sphere $S^{n-1} \subset \RR^n.$ 
In this context, determining conclusive ball-truncated volumes is challenging. Therefore, we selected an appropriate hyperplane $H_x$, where
$x$ is a multiple of $x_1^2+\ldots+x_n^2$, and then derived volume estimates for the intersections of all the cones with $\partial H_x$. For a detailed discussion on the choice of $x$ see \cite[Section 2.2]{KSZ}.
\item 
The proof of Theorem \ref{intro-vrad-of-our-sets} is significantly less demanding than the proof of the corresponding result in \cite{KSZ}, namely Theorem 1.4. To establish the latter, we had to use two distinct inner products: in addition to the $L^2$
  inner product, we also employed the so-called differential inner product on even quartics, which resembles the Frobenius inner product on matrices. The main observations for deriving the estimates relied on the differential inner product, while the connection between the volumes of the sets in both inner products, developed by Blekherman in \cite{Ble04,Ble06}, was crucial. Proving Theorem \ref{intro-vrad-of-our-sets} closely follows the observations regarding the differential inner product from \cite{KSZ}.
\end{enumerate}
\end{remark}

\section{Construction of exceptional doubly nonnegative  and exceptional copositive matrices}
In this section we describe the details of the bootstrap method outlined in Subsection 
\ref{sec: main} to find exceptional doubly nonnegative (e-DNN) and exceptional copositive (e-COP) matrices.
The idea is to find a seed e-DNN matrix of size $5 \times 5$ that is the compression of a multiplication operator $M_f$ for a sum of squares cosine trigonometric polynomial $f$ using a semidefinite optimization program (SDP).
From the seed matrix we then read off the (finitely many) Fourier coefficients of $f.$ 
Finally, we argue that all the larger finite compressions (principal submatrices) of $M_f$ are e-DNN as well.
Using the constructed e-DNN matrices we produce a corresponding family of exceptional copositive matrices.

\subsection{Justification of the construction of a family of e-DNN matrices from a seed e-DNN matrix of size $5\times5$} \label{subsec-findex}
Recall that the function $f$ we are looking for is of the form
$$
f(x) = 1+ 2 \sum_{k=1}^6 a_k\cos(2k\pi x)
$$
with $a_1,\ldots,a_6\geq 0$ (here we immediately set $m=6$ as in \eqref{form-of-f}). Also, in \eqref{candidates-DNN-not-CP}, we defined $A^{(n)} = (A^{(n)}_{jk})_{j,k}$ to be the $n \times n$ principal submatrix of the infinite matrix pertaining to the multiplication operator $M_f^{\cH}$ on $\cH.$ Here $\cH$ is the closed subspace of $L^2[0,1]$    spanned by the $\cos(2k\pi x)$, $k\in \NN_0.$
The technical reasons why we restrict to $\cH$  instead of considering the entire $L^2[0,1]$ are discussed in Remark \ref{rem-sin}. The restriction to matrices of size $n\geq 5$ is clear from the introduction since $\DNN_n=\cp_n$ for $n \leq 4.$

To find the general form of $A^{(n)}$   note that 
\begin{align*}
	A_{jk}^{(n)} &= \int_{0}^{1}f(x)\cos(2(j-1)\pi x) \cos(2(k-1)\pi x)\,dx \quad \text{for}  \ \  j,k=1,\ldots,n,
\end{align*}
where the integration is with respect to the Lebesgue measure on $[0,1].$ Using the well-known trigonometry formula involving the cosine product identity, the products of different cosine functions can be replaced with linear combinations of cosine functions with higher and lower frequency, i.e.,
\begin{align}\label{eq-trigcos}
	\cos(2j\pi x)\cos(2k\pi x) &= \frac{1}{2}\Big( \cos\big(2(j-k)\pi x)\big) + \cos\big(2(j+k)\pi x\big) \Big).
\end{align}
From \eqref{eq-trigcos} it follows  that 
    \begin{equation}
        \label{prod-three-cos}
        \int_0^1 
        \cos(2j\pi x)\cos(2k\pi x)\cos(2\ell\pi x) dx
        =
        \left\{
        \begin{array}{rl}
            \frac{1}{2},&   \text{if }j=\ell, k=0,\\[0.5em]
            \frac{1}{4},&   \text{if }k\neq 0 
                            \text{ and }j\in\{\ell+k,\ell-k\},\\[0.5em]
            0,&             \text{otherwise}.
        \end{array}
        \right.
    \end{equation}
Using \eqref{prod-three-cos} it is now easy to compute that for $A^{(5)}$ to be the $5 \times 5$ compression of $M_f^{\mathcal{H}},$ it must be of the form
\begin{equation}\label{eq-matA}
	A^{(5)}=
	\begin{pmatrix}
		1 & \sqrt{2} a_1 & \sqrt{2} a_2 & \sqrt{2} a_3 & \sqrt{2} a_4 \\
		\sqrt{2} a_1 & a_2+1 & a_1+a_3 & a_2+a_4 & a_3+a_5 \\
		\sqrt{2} a_2 & a_1+a_3 & a_4+1 & a_1+a_5 & a_2 + a_6 \\
		\sqrt{2} a_3 & a_2+a_4 & a_1+a_5 & 1 + a_6 & a_1 \\
		\sqrt{2} a_4 & a_3+a_5 & a_2 + a_6 & a_1 & 1
	\end{pmatrix}.
\end{equation}
Thus demanding that $a_i \geq 0$ for $i=1, \ldots, 6$ 
certifies that $A^{(5)}$ is $\nn$. By the same reasoning
$A^{(n)}$ is $\nn$ for every $n\geq 5$.

Further on, $f$ being of the form \eqref{f-psd-certify}
is equivalent to $f$ being a sum of squares of trigonometric polynomials \cite[Lemma 4.1.3]{Mar08}. 
This implies that all matrices
$A^{(n)}=P_nM_f^\cH P_n$ as in \eqref{candidates-DNN-not-CP}
are PSD. Indeed, suppose
$$
f = \sum_{i=0}^k \big(\underbrace{\sum_{j=0}^{m'}h_{ij}\cos(2j\pi x)}_{h_i}\big)^2
$$
for some $k$ and $h_{ij}\in\RR$. 
Since $f$ and the $h_i$ are in $\mathcal{H},$ clearly $M_f^{\mathcal{H}}$ and the $M^{\mathcal{H}}_{h_i}$ are multiplication operators on $\mathcal{H}$ and 
$$
M_f^{\mathcal{H}} = \sum_{i=1}^{k}\big(M^{\mathcal{H}}_{h_i}\big)^2.
$$
Here each $M^{\mathcal{H}}_{h_i}$ is self-adjoint, from which the claim follows.

Finally, we justify why \eqref{not-cp} implies that
$P_nM_f^{\cH}P_n$ is not $\CP$ for any $n\geq 5$.
Since CP matrices are dual to copositive matrices
in the usual Frobenius inner product,
\eqref{not-cp} certifies that $A^{(5)}$ is not $\CP$. 
Now the equality 
\begin{equation}
    \label{relation-A5-to-An}
    A^{(5)}=P_5(P_nM_f^{\cH}P_n)P_5=P_5A^{(n)}P_5
\end{equation}
for any $n\geq 5$, implies that $A^{(n)}$ is not $\CP$
for any $n\geq 5$.
Indeed, suppose that $A^{(n)}=BB^{\text{T}}$
for some $n\geq 5$ and (not necessarily square) matrix $B$ with nonnegative entries. By \eqref{relation-A5-to-An},
$A^{(5)}=P_5B (P_5B)^{\text{T}}$ and since $P_5$ only has $0,1$ entries, this contradicts $A^{(5)}$ not being $\CP$.

\subsection{Justification of the construction of exceptional COP matrices from exceptional DNN matrices}

It remains to justify our procedure for constructing an exceptional copositive matrix $C$ of size $n \times n$ for any $n \geq 5$ from the obtained e-DNN matrix $A^{(n)}.$ Since SPN matrices are dual to the DNN matrices in the Frobenius inner product, 
the first condition in \eqref{not-SPN},
$$
\langle A^{(n)},C\rangle<0,
$$
implies that $C$ is not SPN. 
On the other hand, the second condition in \eqref{not-SPN},
$$
\big(\sum_{i=1}^nx_i^2\big)^kq_C \text{ is SOS for some }k\in \NN,
$$
is a relaxation of the copositivity of $C$
and it clearly implies that $q_C$ is nonnegative on $\RR^n.$ Whence, $C$ is COP. We remark that in practice, it suffices to consider only $k=1$ or $k=2.$ 

\begin{remark}\label{rem-sin} 
    We explain the reason for restricting $M_f$ to the closed subspace $\mathcal{H}$ of $L^2[0,1]$ generated by the cosine functions. 
	As in Subsection \ref{construction-basic-idea}, with respect to the standard orthonormal basis for $L^2[0,1]$ given by
	\begin{equation}\label{eq-basis}
		1,\ \sqrt{2} \cos\big(2k\pi x\big),\  \sqrt{2} \sin\big(2k\pi  x\big)
	\end{equation}
	for $k \in \NN,$ each multiplication operator $M_f$ for $f \in L^\infty[0,1]$ can be represented by an infinite matrix.
	
	It seems natural to start by considering the entire space $L^2[0,1]$ and compressions $\widetilde P_nM_f\widetilde P_n$ of $M_f$ for some trigonometric polynomial $f$ and $n \geq 2,$ onto the $(2n+1)$-dimensional span $\widetilde \cH_n$ of the functions in \eqref{eq-basis} for $k=1, \ldots, n$. Here $\widetilde P_n:L^2[0,1]\to\widetilde\cH_n$ are orthogonal projections.  
	
	 Suppose that $A$ is the $5\times 5$ compression of $M_f$ and is given with respect to the ordered (orthonormal) basis consisting of the functions
	\begin{align*} 
		1,\sqrt{2}\cos(2\pi x), \sqrt{2}\cos(4\pi x),
		\sqrt{2}\sin(2\pi x), \sqrt{2}\sin(4\pi x).
	\end{align*}
	 Moreover, assume that the corresponding function $f$ has finite Fourier series
	 \begin{equation}\label{eq-fourierf}
	 	f(x) = 1 + 2 \sum_{k=1}^{m} a_k \cos(2k\pi x) + 2 \sum_{k=1}^{m} b_{k} \sin(2k\pi x)
	 \end{equation}
	 for some $m \in \NN$ and real numbers $a_k,b_k$ with $k=1, \ldots,m.$
	 Again, using well-known trigonometry formulas involving product identities, i.e.,
	 \begin{equation}
	 \begin{split}\label{eq-badtrig}
	 	%\cos(2j\pi x)\cos(2k\pi x) &= \frac{1}{2}\Big( \cos\big(2(j-k)\pi x)\big) + \cos\big(2(j+k)\pi x\big) \Big)\\
	 	\sin(2j\pi x)\sin(2k\pi x)& = \frac{1}{2}\Big( \cos\big(2(j-k)\pi x)\big) - \cos\big(2(j+k)\pi x\big) \Big),\\
	 	\cos(2j\pi x)\sin(2k\pi x)& = \frac{1}{2}\Big( \sin\big(2(k-j)\pi x)\big) + \sin\big(2(j+k)\pi x\big) \Big)
	 \end{split}
 	\end{equation}
 	 in addition to \eqref{eq-trigcos},
	 it is easy to compute that for $A$ to be the $5 \times 5$ compression of a multiplication operator $M_f$ for $f$ as in \eqref{eq-fourierf} with $m \geq 4$, it must be of the form
	\begin{equation*}
		A=
		\begin{pmatrix}
			1 & \sqrt{2} a_1 & \sqrt{2} a_2 & \sqrt{2} b_1 & \sqrt{2} b_2  \\
			\sqrt{2} a_1 & a_2+1 & a_1+a_3  & b_2 & b_1+b_3  \\
			\sqrt{2} a_2 & a_1+a_3 & a_4+1 & b_3-b_1 & b_4  \\
			\sqrt{2} b_1 & b_2 & b_3-b_1 & 1-a_2 & a_1-a_3  \\
			\sqrt{2} b_2 & b_1+b_3 & b_4  & a_1-a_3 & 1-a_4 
		\end{pmatrix}.
	\end{equation*}

Note that since we want all the finite-dimensional compressions of $M_f$ to be NN, $f$ needs to have an infinite Fourier series.
Indeed, suppose $f$ has finite Fourier series as in \eqref{eq-fourierf} for some $m \in \mathbb{N}.$ Then for all $j,k$ with $k<j\leq m$ and $m<j+k,$ the $(j+1,k+m+1)$-entry of the compression $\widetilde P_m M_f \widetilde P_m,$ 
\begin{align*}
&\int_{0}^{1}f(x)\, \sqrt{2}\cos(2j\pi x) \,\sqrt{2} \sin(2k\pi x)\,dx = \\ &\int_{0}^{1}f(x)\sin\big(2(k-j)\pi x)\big) \,dx \ +\  \int_{0}^{1}f(x)\sin\big(2(j+k)\pi x)\big) \,dx,
\end{align*}
equals $-a_{j-k}.$ 
Furthermore, we see from \eqref{eq-badtrig} that the Fourier sine coefficients of $f$ must satisfy
$$
b_{j+k} \geq b_{j-k}
$$
for all $k,j.$ 
But the containment $f\in L^2[0,1]$ implies that $b_k = 0$ for all $k.$ 
Hence, $f$ has a Fourier cosine series. To avoid technical difficulties, we thus restrict our attention to $\mathcal{H}.$
\end{remark}

\subsection{Examples}\label{subsec-ex}
\subsubsection{A seed e-DNN $5 \times 5$ matrix}
Let $\epsilon= 1/20.$ Solving the SDP \eqref{eq-sdp} with this parameter and rationalizing the solution \cite{PP08,CKP15} yields the $5\times 5$ compression 
\begin{equation}\label{eq-exA}
\begin{split}
A^{(5)} = 
\begin{pmatrix}
	1 & \frac{16 \sqrt{2}}{27} & \frac{\sqrt{2}}{123} & \frac{1}{147 \sqrt{2}} & \frac{5 \sqrt{2}}{21} \\[0.5em]
	\frac{16 \sqrt{2}}{27} & \frac{124}{123} & \frac{1577}{2646} & \frac{212}{861} & \frac{1205}{8526} \\[0.5em]
	\frac{\sqrt{2}}{123} & \frac{1577}{2646} & \frac{26}{21} & \frac{572}{783} & \frac{1777340 \sqrt{2}-2413803}{3254580} \\[0.5em]
	\frac{1}{147 \sqrt{2}} & \frac{212}{861} & \frac{572}{783} & \frac{1777340 \sqrt{2}+814317}{3254580} & \frac{16}{27} \\[0.5em]
	\frac{5 \sqrt{2}}{21} & \frac{1205}{8526} & \frac{1777340 \sqrt{2}-2413803}{3254580} & \frac{16}{27} & 1 \\
\end{pmatrix}.
\end{split}
\end{equation}
By comparing the above $A^{(5)}$ with the general form \eqref{eq-matA} we read off the Fourier coefficients of the corresponding function $f$ as in \eqref{f-psd-certify}, i.e.,
\begin{align*}
	f(x) =  1 + &\frac{32}{27}\cos(2\pi x) + \frac{2}{123}\cos(4\pi x) + \frac{1}{147}\cos(6\pi x) \\
	&+\frac{10}{21}\cos(8\pi x) + \frac{8}{29}\cos(10\pi x) \\ 
	&+\frac{-2440263 + 1777340 \sqrt{2}}{1627290}\cos(12\pi x).
\end{align*}
This function is indeed SOS, since we have $f=v^{\text{T}}Bv$ for $v$ as in \eqref{vector-v} with $m^\prime = 3$ and 
$$
B=
\begin{pmatrix}
	\frac{9}{22} & \frac{7}{37} & -\frac{3}{22} & -\frac{206923}{5678316} \\[0.5em]
	\frac{7}{37} & \frac{336929}{243540}-\frac{88867 \sqrt{2}}{162729} & \frac{2210}{28971} & \frac{88867}{162729 \sqrt{2}}-\frac{200129}{487080} \\[0.5em]
	-\frac{3}{22} & \frac{2210}{28971} & \frac{46466763-19550740 \sqrt{2}}{35800380} & \frac{4}{29} \\[0.5em]
	-\frac{206923}{5678316} & \frac{88867}{162729 \sqrt{2}}-\frac{200129}{487080} & \frac{4}{29} & \frac{1777340 \sqrt{2}-2440263}{1627290} \\
\end{pmatrix}\succeq 0.
$$

\subsubsection{Exceptional copositive matrix from a DNN matrix}

Now let $\epsilon^\prime = 1/10$ and $k=1.$ From the matrix $A^{(5)}$ in \eqref{eq-exA} we construct an exceptional copositive matrix $C$ as described in Subsection \ref{construction-basic-idea} by solving the feasibility SDP
\begin{equation*}
\begin{split}\label{eq-sdp2}
	&\Tr(CA^{(5)}) = -\epsilon^\prime,\\ 
	&\big(\sum_{i=1}^{n}x_i^2\big) \,q_C = w^{\text{T}}Bw \quad \text{with}\quad B\succeq 0,
\end{split}
\end{equation*}
where $w$ is the vector with all the degree at most 3 words in the variables $x_1, \ldots,x_n.$
Again, after a suitable rationalization, we get an exceptional copositive matrix
$$	
C=
\begin{pmatrix}
	17 & -\frac{91}{5} & \frac{33}{2} & \frac{38}{3} & -\frac{36}{5} \\[0.5em]
	-\frac{91}{5} & \frac{59}{3} & -\frac{53}{4} & 8 & \frac{33}{4} \\[0.5em]
	\frac{33}{2} & -\frac{53}{4} & \frac{39}{4} & -\frac{13}{2} & 8 \\[0.5em]
	\frac{38}{3} & 8 & -\frac{13}{2} & \frac{16}{3} & -\frac{13}{3} \\[0.5em]
	-\frac{36}{5} & \frac{33}{4} & 8 & -\frac{13}{3} & \frac{1373628701}{353935575} \\
\end{pmatrix}.
$$

%%%%%%%%%%%%%%%%%%%%%%%%%%%%%%

\section{Quantifying the gap between $\COP_n$ and $\CP_n$}

In this section we prove our second main result (Theorem \ref{intro-vrad-of-our-sets}) on the estimates of volume radii of the cones from Definition \ref{def-cones}:

\begin{theorem}
    \label{vrad-of-our-sets} 
    We have that
    \begin{align*}
        %\label{intro-vrad-asymptotics}
            \frac{1}{8\sqrt{2}}
            &\leq 
            \vrad(\CP^{(B_n)}_n)
            \leq
            \vrad(\nn^{(B_n)}_n)
            =\frac{1}{2}
            \leq 
            \vrad(\COP^{(B_n)}_n)
            \leq 
            1.
    \end{align*}
    In particular,
    $$
    \frac{1}{8\sqrt{2}}
    \leq
    \frac
    {\vrad(K^{(B_n)}_n)}{\vrad(\COP^{(B_n)}_n)}
    \leq 
    1,
    $$
    where $K\in \{\CP,\DNN,\psd,\nn,\SPN\}$.
\end{theorem}

\begin{proof}
First we establish two claims.\\

\noindent \textbf{Claim 1.} 
$
\vrad{(\nn_n^{(B_n)})}=\frac{1}{2}.
$\\

\noindent\textit{Proof of Claim 1.}
    We first show that $B_n$ is a disjoint union of 
    $2^{\dim \Sym_n}$ 
    copies of 
    $\nn_n^{(B_n)}$ and hence
    $$\Vol(\nn_n^{(B_n)})=2^{-\dim \Sym_n}
    \cdot
    \Vol(B_n).$$
    Indeed,
    for 
    $A=[a_{ij}]_{ij}\in \{0,1\}^{n\times n}\cap \Sym_n$ let
    $S_{A}:=B_n\cap H_{A}$,
    where 
    $$H_{A}=
    \{
    [b_{ij}]_{ij}\in \Sym_n\colon
    (-1)^{a_{ij}+1} \cdot b_{ij}\geq 0
    \text{ for all } i,j
    \}.
    $$
    The ones (zeros resp.) in the matrix $A$ thus determine the entries $b_{ij}$ that have positive (negative resp.) sign.
    Note that  $$ B_n=\bigsqcup_{A\in \{0,1\}^{n\times n}\cap \Sym_n} S_A$$ and $\Vol S_A=\Vol \nn_n^{(B_n)}$ for every $A\in \{0,1\}^{n\times n}\cap \Sym_n.$ Hence
    $$ \Vol B_n=\sum_{A\in \{0,1\}^{n\times n}\cap \Sym_n}\Vol S_A = 2^{\dim \Sym_n}\Vol(\nn_n^{(B_n)}),$$
    which immediately proves Claim 1.
    $\hfill\square$\\
    
    \iffalse
    Let $A=[a_{ij}]_{i,j}\in B_n$.
    We have that
    $A=A_+-A_-$, where $A_\pm=[a_{ij}^\pm]_{i,j}$, 
    $a_{ij}^+:=\max(a_{ij},0)$
    and
    $a_{ij}^-:=|\min(a_{ij},0)|$.
    Note that $\|A_{\pm}\|\leq \|A\|\leq 1$
    and hence $A_{\pm}\in \nn_{n}^{B_n}$.
    Therefeore $A\in \symm(\nn^{(B_n)}_n)$,
    which proves Claim 1.$\hfill\blacksquare$
    \fi

Next, denote by
$$\symm(\cp^{(B_n)}_n):=
\cp^{(B_n)}_n-\cp^{(B_n)}_n
=\{U-V\colon U,V\in \cp^{(B_n)}_n\}$$
the difference body of $\cp^{(B_n)}_n$.\\

\noindent \textbf{Claim 2.}
$
\cp^{(B_n)}_n\subseteq
\nn^{(B_n)}_n
\subseteq 
\sqrt{2}\symm(\cp^{(B_n)}_n)$.\\

\noindent\textit{Proof of Claim 2.}
    The left inclusion
    is clear. 
    To prove the right inclusion
    it suffices to prove that every extreme point of $\nn^{(B_n)}_{n}$
    is contained in $\symm(\cp^{(B_n)}_n)$.
    Note that the extreme points of $\nn^{(B_n)}_{n}$
    are of two types:
    \begin{align}
        \label{ext-nn-type1}
            E_{ii}&\qquad
            \text{for some }i=1,\ldots,n,\\
        \label{ext-nn-type2}
            \displaystyle 
            \frac{1}{\sqrt{2}}
            (E_{ij}+E_{ji})&\qquad 
            \text{for some }i,j=1,\ldots,n,\ i\neq j,
    \end{align}
    where $E_{ij}$ are the standard matrix basis, i.e., the only nonzero entry of $E_{ij}$ is 1 at position $(i,j)$.
    The extreme points of the form \eqref{ext-nn-type1} clearly belong to
    $
    \cp^{(B_n)}_n
    \subseteq
    \symm(\cp^{(B_n)}_n)
    $.
    It remains to study the extreme points of the form \eqref{ext-nn-type2}.
  Note that
$$
F := E_{ij}+E_{ji} +E_{ii} + E_{jj} = xx^\text{T}  \in  2\cp_n^{(B_n)},
$$
where $x \in \RR^n$ is a vector with zeros except at positions $i$ and $j,$ where it has ones. Hence
\begin{align*}
    \frac{1}{\sqrt{2}}
        (E_{ij}+E_{ji})
        &=  \frac{1}{\sqrt{2}} F - \frac{1}{\sqrt{2}}\left(E_{ii} + E_{jj}\right) \in 
        %\frac{1}{\sqrt{2}} \symm(\cp_n^{(B_n)}) \subseteq 
        \sqrt{2}\symm(\cp_n^{(B_n)}),
\end{align*}
 which concludes the proof of Claim 2. 
    $\hfill\square$
 \\
 
 By the Rogers-Shepard inequality \cite[Theorem 1]{RS57}
 we have that 
 \begin{align}
      \label{RS-ineq-cp}
\vrad(\symm(\cp_n^{(B_n)}))
 &\leq 
 4\vrad(\cp_n^{(B_n)}).
 \end{align}
 By \eqref{RS-ineq-cp} and Claims $1$ and $2,$ it follows that
\begin{align}
 \begin{split}
     \label{RS-ineq-cp-combined}
     \frac{1}{8\sqrt{2}}
 &\leq \vrad(\cp_n^{(B_n)})
 \leq \frac{1}{2}.
  \end{split}
 \end{align}
The statements in Theorem \ref{vrad-of-our-sets} now follow from Claim 1 and \eqref{RS-ineq-cp-combined}.
\end{proof}

\end{document}